\documentclass{amsart}
\IfFileExists{lmodern.sty}{\usepackage[T1]{fontenc}\usepackage{lmodern}}{}
\usepackage{amsmath,amssymb,amsthm}
\usepackage{tikz}
\usetikzlibrary{decorations.pathreplacing,arrows.meta}
\usepackage{hyperref}
\numberwithin{equation}{section}

\theoremstyle{plain}
\newtheorem{theorem}{Theorem}[section]
\newtheorem{lemma}[theorem]{Lemma}
\newtheorem{proposition}[theorem]{Proposition}
\newtheorem{corollary}[theorem]{Corollary}
\newtheorem{conjecture}[theorem]{Conjecture}
\theoremstyle{remark}
\newtheorem{remark}[theorem]{Remark}
\newtheorem{question}[theorem]{Question}

\newcommand{\F}{\mathcal F}
\newcommand{\cP}{\mathcal P}
\newcommand{\cU}{\mathcal U}
\newcommand{\T}{\mathbb T}
\newcommand{\diam}{\operatorname{diam}}
\newcommand{\wn}{\operatorname{wn}}
\newcommand{\id}{\mathrm{id}}
\newcommand{\inte}{\operatorname{int}}

\title{The circle as a topological fractal}
\author{Benjamin Vejnar}
\address{Faculty of Mathematics and Physics, Charles University, Prague, Czech Republic}

\begin{document}

\begin{abstract}
We prove that no family of two continuous self-maps witnesses that the circle is a topological fractal, answering a question of Karasová and the present author. Since three maps are known to suffice, this bound is optimal. In contrast, for every $\varepsilon>0$ there are two continuous self-maps of the circle, depending on $\varepsilon$, whose images cover the circle and an integer $N$ such that every composition of $N$ of them has image of diameter less than $\varepsilon$. Thus two maps suffice at any prescribed scale, but no fixed pair works at all scales.
\end{abstract}

\maketitle

\section{Introduction}

Let $(X,d)$ be a nonempty compact metric space and let $\F$ be a nonempty finite family
of continuous self-maps of $X$. Put
\[
\F^0=\{\id_X\},\qquad \F^n=\{f_1\circ\dots\circ f_n : f_i\in\F\}\qquad(n\ge1).
\]
We say that $\F$ is \emph{topologically contractive} if for every $\varepsilon>0$ there
exists $n\ge1$ such that
\begin{equation} \label{eq:1.1}
\diam_d w(X)<\varepsilon\qquad\text{for every } w\in\F^n.
\end{equation}
A compact metric space $X$ is called a \emph{topological fractal} if there is a finite
topologically contractive family $\F$ of continuous self-maps such that
\begin{equation*}%\label{eq:1.2}
X=\bigcup_{f\in\F}f(X).
\end{equation*}
The maps in such a family are called \emph{witnessing maps}. This terminology is used,
in particular, in~\cite{KarasovaVejnar}; see also the general theory of contractive function systems
in~\cite{BanachKubisNovosadNowakStrobin}.

Although the definition is written using a metric, it depends only on the topology of
$X$. Indeed, if $d$ and $\rho$ are two compatible metrics on the compact space $X$, then
the identity maps between $(X,d)$ and $(X,\rho)$ are uniformly continuous. Hence, given
$\varepsilon>0$ in the metric $\rho$, there is $\delta>0$ such that every subset of
$d$-diameter less than $\delta$ has $\rho$-diameter less than $\varepsilon$. Applying
\eqref{eq:1.1} with $\delta$ shows that topological contractivity for $d$ implies
topological contractivity for $\rho$, and conversely. Equivalently, one can formulate the
condition without a metric: for every open cover $\cU$ of $X$ there is $n$ such that
every $w(X)$, $w\in\F^n$, is contained in one member of $\cU$. For compact metrizable
spaces the equivalence follows from the Lebesgue number lemma.

The notion is motivated by the classical theory of iterated function systems.
Hutchinson's theorem associates a unique compact invariant set to a finite family of
contractions on a complete metric space~\cite{Hutchinson}. Hata studied the topology of
connected self-similar sets and the relation with Peano continua~\cite{Hata}. In the
topological setting the long-standing converse problem takes the following form;
following~\cite{KarasovaVejnar}, we refer to it as the Hata conjecture.

\begin{conjecture}[Hata]\label{conj:hata}
Every Peano continuum is a topological fractal.
\end{conjecture}

Here a \emph{Peano continuum} is a compact, connected, locally connected metrizable
space. Hata observed that every connected topological fractal is a Peano continuum. The
Conjecture~\ref{conj:hata} remains open. Several partial results are known. In particular every Peano
continuum with uncountably many local cut points is a topological fractal~\cite{KarasovaVejnar}.

For a topological fractal $X$, its \emph{witnessing number}, denoted by $\wn(X)$, is the
least cardinality of a family of witnessing maps. This quantity asks not merely whether
$X$ is a topological fractal, but how many self-maps are required to exhibit its fractal
structure. Karasová and the present author proved that three witnessing maps suffice for every Peano
continuum with uncountably many local cut points, and asked whether the bound is already
optimal for the simplest relevant example, a simple closed curve $S$~\cite{KarasovaVejnar}. In
particular, their result gives
\[
\wn(S)\le3.
\]

Whether two witnessing maps exist is a subtle question. It was noted on MathOverflow \cite{MO_circle} by the author of this paper that for every prescribed $\varepsilon>0$
one can choose two maps which cover the circle and whose compositions at a sufficiently high common level have images smaller than $\varepsilon$. A precise formulation is in Proposition~\ref{prop:finite}. The maps, however, depend on the prescribed $\varepsilon$. The
main result shows that this dependence cannot be removed.

\begin{theorem}\label{thm:main}
Let $S$ be a space homeomorphic to the circle. If continuous maps $f,g\colon S\to S$
satisfy
\[
S=f(S)\cup g(S),
\]
then $\{f,g\}$ is not topologically contractive.
\end{theorem}

As a consequence, we answer Question 27 from \cite{KarasovaVejnar}.

\begin{corollary}\label{cor:wn}
The witnessing number of the circle is equal to three.
\end{corollary}

\begin{proof}
The upper bound $\wn(S)\le3$ is simple and it is contained in~\cite{KarasovaVejnar}. A single witnessing map on a
nondegenerate compact space would have to be surjective, and hence all of its iterates
would be surjective, so one map cannot be topologically contractive.
Theorem~\ref{thm:main} excludes two witnessing maps. Therefore $\wn(S)\geq3$.
Consequently $\wn(S)= 3$.
\end{proof}

The proof of Theorem~\ref{thm:main} is contained in Section \ref{sectionmain} and it is based on Lemmas from Sections \ref{sectionmetric}, \ref{sectionCovering}, and \ref{sectionbounded}. We first replace the metric by an equivalent metric in which the two maps are nonexpanding.
We then associate to each arc the minimum number of connected sets of small diameter
required to cover it. If we suppose for contradiction that a topologically contracting two-map covering of the circle exists, we find a uniform bound on how much
the covering number can decrease under one of the maps. The shrinking iterates of
that map force the loss of a fixed nondegenerate end arc, producing a contradiction.

\section{A two-map approximation}

In this section, we use the following metric representation of the circle. Equip
\[
\T=\mathbb R/\mathbb Z
\]
with the normalized intrinsic metric
\begin{equation*} %\label{eq:2.1}
d_\T([x],[y])=\min_{k\in\mathbb Z} |x-y-k|.
\end{equation*}

\begin{proposition}\label{prop:finite}
For every $\varepsilon>0$ there are continuous maps $f,g\colon\T\to\T$ and an integer
$N\ge1$ such that
\begin{equation*} %\label{eq:2.2}
f(\T)\cup g(\T)=\T
\end{equation*}
and
\begin{equation*} %\label{eq:2.3}
\diam_{d_\T}w(\T)<\varepsilon\qquad\text{for every } w\in\{f,g\}^N.
\end{equation*}
The maps may be chosen so that $g=f+\tfrac12\pmod 1$.
\end{proposition}

\begin{proof}
Fix an integer $L\ge3$, and set
\[
\delta=\frac1{2L},\qquad r=\frac{L-2}{L-1}<1.
\]
Using representatives in $[0,1]$, define $f_L\colon\T\to\T$ by
\begin{equation*}%\label{eq:2.4}
f_L(x)=
\begin{cases}
\dfrac12-\dfrac1L+2x, & 0\le x\le\delta,\\[8pt]
\dfrac12-r(x-\delta), & \delta\le x\le\dfrac12,\\[8pt]
\dfrac1L-2\Bigl(x-\dfrac12\Bigr), & \dfrac12\le x\le\dfrac12+\delta,\\[8pt]
r\Bigl(x-\dfrac12-\delta\Bigr), & \dfrac12+\delta\le x\le1.
\end{cases}
\end{equation*}
The values at the breakpoints agree, and
\[
f_L(0)=f_L(1)=\frac12-\frac1L,
\]
so $f_L$ is a continuous circle map; see Figure~\ref{fig:fL}. Put
\[
g_L(x)=f_L(x)+\frac12\pmod 1.
\]
The breakpoint values give
\begin{equation*}
f_L(\T)=\Bigl[0,\frac12\Bigr],\qquad g_L(\T)=\Bigl[\frac12,1\Bigr].
\end{equation*}

\begin{figure}[htb]
\centering
\begin{tikzpicture}[x=5.8cm,y=5.8cm,font=\small]
  \fill[black!10] (0.5,0.5) rectangle (0.7,0.7);
  \foreach \x in {0.1,0.5,0.6} \draw[black!30,densely dotted] (\x,0) -- (\x,1);
  \foreach \y in {0.2,0.3,0.7} \draw[black!30,densely dotted] (0,\y) -- (1,\y);
  \draw[black!55,dashed] (0,0.5) -- (1,0.5);
  \draw (0,0) rectangle (1,1);
  % graphs (L = 5: delta = 1/10, r = 3/4)
  \draw[very thick] (0,0.3) -- (0.1,0.5) -- (0.5,0.2) -- (0.6,0) -- (1,0.3);
  \draw[very thick,dashed] (0,0.8) -- (0.1,1) -- (0.5,0.7) -- (0.6,0.5) -- (1,0.8);
  \foreach \x/\y in {0/0.3,0.1/0.5,0.5/0.2,0.6/0,1/0.3} \fill (\x,\y) circle (1.1pt);
  \foreach \x/\y in {0/0.8,0.1/1,0.5/0.7,0.6/0.5,1/0.8} \draw[fill=white] (\x,\y) circle (1.1pt);
  % tick labels
  \foreach \x/\l in {0/0,0.1/\delta,0.5/\tfrac12,0.6/\tfrac12+\delta,1/1}
    \node[below] at (\x,-0.004) {$\l$};
  \foreach \y/\l in {0/0,0.2/\tfrac1L,0.3/\tfrac12-\tfrac1L,0.5/\tfrac12,0.7/\tfrac12+\tfrac1L,1/1}
    \node[left] at (-0.004,\y) {$\l$};
  % graph labels
  \node at (0.24,0.455) {$f_L$};
  \node at (0.33,0.92) {$g_L$};
  \node[font=\scriptsize] at (0.41,0.40) {slope $-r$};
  \node[font=\scriptsize] at (0.87,0.085) {slope $r$};
  \node[font=\scriptsize] at (0.685,0.665) {$E_L\times E_L$};
  % braces for the images
  \draw[decorate,decoration={brace,amplitude=4pt}] (1.02,0.5) -- (1.02,0)
     node[midway,right=5pt] {$f_L(\T)$};
  \draw[decorate,decoration={brace,amplitude=4pt}] (1.02,1) -- (1.02,0.5)
     node[midway,right=5pt] {$g_L(\T)$};
  % E_L on the x-axis
  \draw[thick] (0.5,-0.115) -- (0.7,-0.115) node[midway,below] {$E_L$};
  \draw[thick] (0.5,-0.10) -- (0.5,-0.13) (0.7,-0.10) -- (0.7,-0.13);
\end{tikzpicture}
\caption{The maps $f_L$ and $g_L$, drawn for $L=5$.}
\label{fig:fL}
\end{figure}
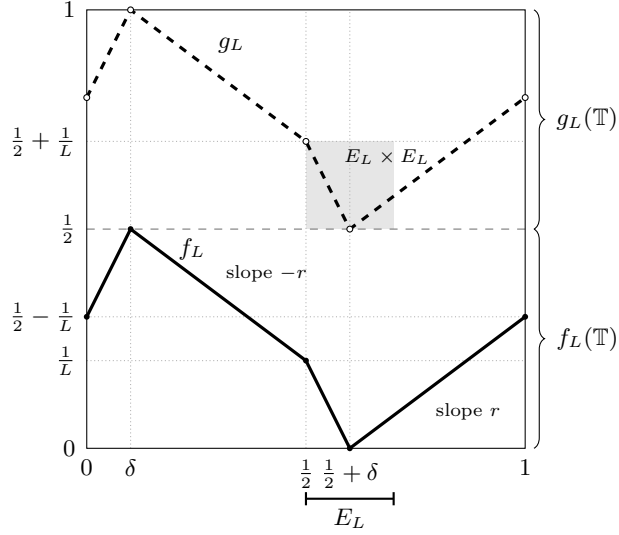

Let $J$ be a subarc of either of the two semicircles \([0,1/2]\) or \([1/2,1]\). On
each of these semicircles, $f_L$ has one branch of domain length $\delta$ and absolute slope $2$,
and one branch of absolute slope $r$. The two branches meet at an extremum. Hence
\begin{equation}\label{eq:2.6}
\diam f_L(J)\le\max\Bigl\{\frac1L,\ r\diam J\Bigr\}.
\end{equation}
Indeed, the variation on the short branch is at most $2\delta=1/L$, while on the other
branch it is at most $r\diam J$. If $J$ meets both branches, the total image diameter is
the larger of the two one-sided variations from the common extremum. Since $g_L$ differs
from $f_L$ by an isometry of the target, \eqref{eq:2.6} also holds with $g_L$ in place of
$f_L$.

For $n\ge1$ set
\[
D_n=\max\{\diam w(\T) : w\in\{f_L,g_L\}^n\}.
\]
Clearly $D_1=1/2$. Every cylinder image of positive level is a connected
subarc of one of the two semicircles. Thus
\[
D_{n+1}\le\max\Bigl\{\frac1L,\ rD_n\Bigr\},
\]
and induction gives
\begin{equation}\label{eq:2.7}
D_n\le\max\Bigl\{\frac1L,\ \frac12r^{\,n-1}\Bigr\}.
\end{equation}

Given $\varepsilon>0$, choose
\[
L>\max\Bigl\{2,\frac1\varepsilon\Bigr\}.
\]
Then $1/L<\varepsilon$. Since $0<r<1$, choose $N$ so large that
$\frac12r^{N-1}<\varepsilon$. Equation \eqref{eq:2.7} gives $D_N<\varepsilon$.
\end{proof}

\begin{remark}\label{rem:quant}
Proposition~\ref{prop:finite} has the quantifier order
\[
\forall\varepsilon>0\ \exists f,g,N,
\]
whereas two witnessing maps would require one fixed pair satisfying
\[
\exists f,g\ \forall\varepsilon>0\ \exists N.
\]
The distinction is essential. For the maps in the proof, put
\[
E_L=\Bigl[\frac12,\frac12+\frac1L\Bigr].
\]
Then \(
g_L(E_L)=E_L
\) (see Figure \ref{fig:fL}).
Consequently $E_L\subseteq g_L^n(\T)$ for every $n$, so the fixed pair $\{f_L,g_L\}$ is
not topologically contractive. Theorem~\ref{thm:main} shows that no other fixed covering
pair can be topologically contractive either.
\end{remark}

\section{A common nonexpanding metric}\label{sectionmetric}

For a compatible metric $d$ on a compact space $X$ and a finite family $\F$ of continuous self-maps of $X$, write
\begin{equation}\label{eq:3.1}
D_n(d,\F)=\max_{w\in\F^n}\diam_d w(X).
\end{equation}
The sequence $D_n(d,\F)$ is clearly nonincreasing
and topological contractivity of $\F$ is equivalent to $D_n(d,\F)\to0$.
The following lemma is essentially a consequence of the proof of \cite[Theorem 6.3]{BanachKubisNovosadNowakStrobin}. For the sake of completeness, we include it here with the proof.

\begin{lemma}\label{lem:metric}
Let $(X,d)$ be a nonempty compact metric space and let $\F$ be a finite topologically
contractive family of continuous self-maps of $X$. Put $W=\bigcup_{n\ge0}\F^n$ and
define
\begin{equation}\label{eq:3.2}
\rho(x,y)=\sup_{w\in W}d\bigl(w(x),w(y)\bigr).
\end{equation}
Then $\rho$ is a compatible metric on $X$ and every $h\in\F$ is nonexpanding with respect to
$\rho$.
%, and \begin{equation}\label{eq:3.3} \max_{u\in\F^n}\diam_\rho u(X)\le D_n(d,\F)\qquad(n\ge1).
%\end{equation}
\end{lemma}

\begin{proof}
Compactness makes $d$ bounded, so the supremum in \eqref{eq:3.2} is finite. Symmetry and
the triangle inequality follow from the corresponding properties of $d$. Since
$\id_X\in W$, we have $\rho\ge d$. %, and $\rho$ separates points.

To prove compatibility of $\rho$ and $d$, fix $\eta>0$. Choose $N$ such that $D_n(d,\F)<\eta/2$ for all
$n\ge N$. %Words of length at least $N$ contribute less than $\eta/2$ to \eqref{eq:3.2}.
The finitely many $w\in \bigcup_{n=0}^{N-1}\F^n$ are uniformly continuous, so there is $\delta>0$ such
that
\[
d(x,y)<\delta\quad\Longrightarrow\quad d\bigl(w(x),w(y)\bigr)<\eta/2
\]
for every such $w$. Hence $d(x,y)<\delta$ implies
$\rho(x,y)\le\eta/2<\eta$. Together with $d\le\rho$, this proves that $d$ and $\rho$
generate the same topology.

For $h\in\F$ we have $\{w\circ h : w\in W\}\subseteq W$, and therefore
\begin{equation*} %\label{eq:3.4}
\rho\bigl(h(x),h(y)\bigr)=\sup_{w\in W}d\bigl((w\circ h)(x),(w\circ h)(y)\bigr)
\le\rho(x,y).
\end{equation*}
Thus, every $h\in\F$ is nonexpanding with respect to $\rho$.
%Finally, for $u\in\F^n$,
%\[ \diam_\rho u(X)=\sup_{w\in W}\diam_d\bigl((w\circ u)(X)\bigr)\le D_n(d,\F), \]
%because $w\circ u$ has a representation of length at least $n$ and the sequence in
%\eqref{eq:3.1} is nonincreasing. This proves \eqref{eq:3.3}.
\end{proof}

\section{Covering numbers}\label{sectionCovering}

Throughout this section, $(S,\rho)$ is a compact metric space homeomorphic to the circle.
A \emph{closed arc} is a subset homeomorphic to $[0,1]$, and a \emph{subcontinuum} is a
nonempty compact connected subset. Every subcontinuum of $S$ is a singleton, a closed
arc, or $S$ itself.

Let $\varepsilon>0$ and let $K\subseteq S$ be nonempty. An \emph{$\varepsilon$-cover} of
$K$ is a finite family of connected subsets of $S$, each of $\rho$-diameter at most
$\varepsilon$, whose union contains $K$. Define
\begin{equation*} %\label{eq:4.1}
N_\varepsilon(K)=\min\{|\mathcal C| : \mathcal C\text{ is an $\varepsilon$-cover of }K\}.
\end{equation*}
Since $S$ is a continuous image of $[0,1]$, uniform continuity shows that $S$, and hence
every $K$, has an $\varepsilon$-cover, so the minimum is well defined. Clearly
$N_\varepsilon(\{x\})=1$. An $\varepsilon$-cover of $K$
with $N_\varepsilon(K)$ members is called \emph{optimal}.
The closure of a connected set is connected and has the same diameter, so we may and do
assume that the members of $\varepsilon$-covers are closed. Then every member is a
singleton, a closed arc, or $S$. Notice that the condition concerns the diameters of the
covering sets only; no length structure is assumed for $\rho$.

\begin{lemma}\label{lem:restr}
Let $\varepsilon>0$. If $K,K_1,\dots,K_m$ are nonempty subsets of $S$ with
$K\subseteq K_1\cup\dots\cup K_m$, then
\begin{equation*} %\label{eq:4.2}
N_\varepsilon(K)\le\sum_{i=1}^mN_\varepsilon(K_i).
\end{equation*}
\end{lemma}

\begin{proof}
The union of optimal $\varepsilon$-covers of the sets $K_i$ is an $\varepsilon$-cover of
$K$.
\end{proof}

\begin{lemma}\label{lem:nonexp}
Let $h\colon S\to S$ be continuous and nonexpanding. If $K\subseteq S$ is nonempty, then
\begin{equation*} %\label{eq:4.3}
N_\varepsilon(h(K))\le N_\varepsilon(K)\qquad\text{for every }\varepsilon>0.
\end{equation*}
\end{lemma}

\begin{proof}
If $\mathcal C$ is an optimal $\varepsilon$-cover of $K$, then the sets $h(C)$,
$C\in\mathcal C$, are connected, their union contains $h(K)$, and
$\diam_\rho h(C)\le\diam_\rho C\le\varepsilon$.
\end{proof}

\begin{lemma}\label{lem:cut}
Let $\cP$ be a finite family of closed arcs in $S$ with pairwise disjoint interiors such
that $\bigcup\cP$ is connected. Then
\begin{equation*} %\label{eq:4.4}
\sum_{I\in\cP}N_\varepsilon(I)\le N_\varepsilon\Bigl(\bigcup\cP\Bigr)+2|\cP|\qquad
\text{for every }\varepsilon>0.
\end{equation*}
\end{lemma}
 
\begin{proof}
Fix $\varepsilon>0$, and put
$M=\bigcup\cP$. Let $\mathcal C$ be an optimal
$\varepsilon$-cover of $M$ consisting of closed sets.
Every point of $M$ belongs to at most two members of $\mathcal C$.
Fix an orientation of $S$, and denote by $a_I$ the initial
endpoint of each $I\in\cP$. For every $C\in\mathcal C$ it can be easily seen that
\[
\bigl|\{I\in\cP:C\cap I\ne\emptyset\}\bigr|
\le 1+\bigl|\{I\in\cP:a_I\in C\}\bigr|.
\]
%To see this, traverse $C$ in the chosen orientation.
%Except possibly for the first arc of $\cP$ encountered,
%every arc that meets $C$ is entered through its initial
%endpoint. This also applies to intersections consisting
%only of an endpoint; if $C=S$, the inequality is immediate.

Now put
$\mathcal C_I=\{C\in\mathcal C:C\cap I\ne\emptyset\}$.
Since $\mathcal C_I$ covers $I$, Lemma \ref{lem:restr}, double counting and the
bound above give
\[
\begin{aligned}
\sum_{I\in\cP}N_\varepsilon(I)
&\le \sum_{I\in\cP}|\mathcal C_I|\\
&= \sum_{C\in\mathcal C}
   \bigl|\{I\in\cP:C\cap I\ne\emptyset\}\bigr|\\
&\le |\mathcal C|
   +\sum_{I\in\cP}\bigl|\{C\in\mathcal C:a_I\in C\}\bigr|\\
&\le N_\varepsilon(M)+2|\cP|.
\end{aligned}
\]
\end{proof}

\begin{lemma}\label{lem:end}
If $a,b$ are the distinct endpoints of a closed arc $L\subseteq S$, then
\begin{equation*} %\label{eq:4.6}
\rho(a,b)\le\varepsilon N_\varepsilon(L)\qquad\text{for every }\varepsilon>0.
\end{equation*}
In particular, $N_\varepsilon(L)\to\infty$ as $\varepsilon\downarrow0$.
\end{lemma}

\begin{proof}
Let $\mathcal C$ be an optimal $\varepsilon$-cover of $L$, and fix $C_a\in\mathcal C$ with
$a\in C_a$. Since $L$ is connected, we can find a shortest sequence $C_1=C_a,C_2,\dots,C_k$ with $b\in C_k$ and points $x_i\in C_i\cap C_{i+1}$ for
$1\le i<k$. Put $x_0=a$ and $x_k=b$. Then $x_{i-1},x_i\in C_i$ for every $i$, and
therefore
\[
\rho(a,b)\le\sum_{i=1}^k\rho(x_{i-1},x_i)\le \sum_{i=1}^k\diam C_i\le k\varepsilon\le\varepsilon N_\varepsilon(L).
\]
The final assertion follows from $\rho(a,b)>0$.
\end{proof}

\section{The bounded-loss estimate}\label{sectionbounded}

\begin{lemma}\label{lem:loss}
Let $f,g\colon S\to S$ be nonexpanding maps of a metric circle $(S,\rho)$. Suppose that
\[
A=f(S),\qquad B=g(S)
\]
are nondegenerate proper closed arcs and that $A\cup B=S$. Choose $u,v\in S$ such that
$f(u)$ and $f(v)$ are the two endpoints of $A$. Then every nondegenerate closed arc
$J\subseteq S$ with
\[
\inte_S(J)\cap\{u,v\}=\emptyset
\]
satisfies
\begin{equation}\label{eq:5.1}
0\le N_\varepsilon(J)-N_\varepsilon(f(J))\le 12\qquad\text{for every }\varepsilon>0.
\end{equation}
\end{lemma}

\begin{proof}
Fix $\varepsilon>0$ and abbreviate $N_\varepsilon$ to $N$. By Lemma~\ref{lem:nonexp},
\[
\Delta=N(J)-N(f(J))
\]
is nonnegative.

The points $u,v$ are distinct and divide $S$ into two closed arcs $H_1,H_2$. For each
$i$, the set $f(H_i)$ is a connected subset of $A$ containing both endpoints of $A$.
Hence
\begin{equation}\label{eq:5.2}
f(H_1)=f(H_2)=A.
\end{equation}

Let $Z$ be the set consisting of $u$, $v$ and the two endpoints of $J$, and 
let $\cP$ consist of the closures of components of $S\setminus Z$. Thus $\cP$ is
a family consisting of at
most four closed arcs with disjoint interiors, which have endpoints in $Z$ (see Figure~\ref{fig:loss}). Call the elements of $\cP$ pieces. Since neither $u$
nor $v$ belongs to $\inte_S(J)$, the arc $J$ itself is one of the pieces. The pieces
contained in $H_1$ cover $H_1$, so their images under $f$ cover $A$ by \eqref{eq:5.2}, and the same
is true for the pieces contained in $H_2$. Lemma~\ref{lem:restr}, applied on the two
sides, gives
\[
2N(A)=N(f(H_1))+N(f(H_2))\le\sum_{I\in\cP}N(f(I)).
\]
For every piece $I$, Lemma~\ref{lem:nonexp} gives $N(f(I))\le N(I)$, and on the
distinguished piece $J$ the loss is $\Delta$. Hence, by Lemma~\ref{lem:cut} with $M=S$,
since $\cP$ has at most four pieces,
\begin{equation}\label{eq:5.3}
2N(A)\le\sum_{I\in\cP}N(f(I))\le\sum_{I\in\cP}N(I)-\Delta\le N(S)+2\cdot 4-\Delta=N(S)+8-\Delta.
\end{equation}

\begin{figure}[htb]
\centering
\begin{tikzpicture}[>=Latex,font=\small]
  \def\rad{1.8}
  % ---------- domain ----------
  \begin{scope}
    \draw[black!40] (0,0) circle (\rad);
    \draw[line width=0.9pt] (160:\rad) arc (160:340:\rad);   % H_2
    \draw[line width=0.9pt] (340:\rad) arc (340:400:\rad);   % I_2
    \draw[line width=0.9pt] (120:\rad) arc (120:160:\rad);   % I_1
    \draw[line width=2.2pt] (40:\rad) arc (40:120:\rad);     % J
    \fill (160:\rad) circle (1.7pt);
    \fill (340:\rad) circle (1.7pt);
    \draw[fill=white] (40:\rad) circle (1.5pt);
    \draw[fill=white] (120:\rad) circle (1.5pt);
    \node at (160:\rad+0.33) {$u$};
    \node at (340:\rad+0.33) {$v$};
    \node at (80:\rad+0.3) {$J$};
    \node at (140:\rad+0.4) {$I_1$};
    \node at (10:\rad+0.4) {$I_2$};
    \node at (250:\rad+0.42) {$H_2$};
    \node at (90:\rad-0.85) {$H_1=I_1\cup J\cup I_2$};
    \node at (270:\rad-0.55) {$S$};
  \end{scope}
  \draw[->,thick] (\rad+1.0,0) -- (\rad+2.3,0) node[midway,above] {$f$};
  % ---------- target ----------
  \begin{scope}[xshift=7.7cm]
    \draw[black!40] (0,0) circle (\rad);
    \draw[line width=2.2pt] (-30:\rad) arc (-30:210:\rad);     % A = f(S)
    \draw[line width=0.9pt] (210:\rad) arc (210:330:\rad);     % A^*
    \draw[dashed,thick] (180:\rad+0.3) arc (180:360:\rad+0.3); % B = g(S)
    \draw[line width=0.9pt] (100:\rad-0.35) arc (100:210:\rad-0.35);  % f(I_1)
    \draw[line width=0.9pt] (-30:\rad-0.35) arc (-30:80:\rad-0.35);   % f(I_2)
    \draw[line width=2.2pt] (55:\rad-0.6) arc (55:130:\rad-0.6);      % f(J)
    \fill (210:\rad) circle (1.7pt);
    \fill (330:\rad) circle (1.7pt);
    \node at (90:\rad+0.35) {$A=f(S)$};
    \node at (270:\rad+0.65) {$B=g(S)$};
%    \node at (270:\rad-0.35) {$A^*$};
    \node at (213:\rad+0.72) {$f(u)$};
    \node at (327:\rad+0.72) {$f(v)$};
    \node at (160:\rad-0.88) {$f(I_1)$};
    \node at (20:\rad-0.88) {$f(I_2)$};
    \node at (92:\rad-0.95) {$f(J)$};
  \end{scope}
\end{tikzpicture}
\caption{Left: $\cP=\{I_1,J,I_2,H_2\}$. Right: $A=f(S)=f(H_1)=f(H_2)$ with endpoints $f(u)$, $f(v)$, the images
$f(I_1)$, $f(J)$, $f(I_2)$ cover $A$.}
\label{fig:loss}
\end{figure}
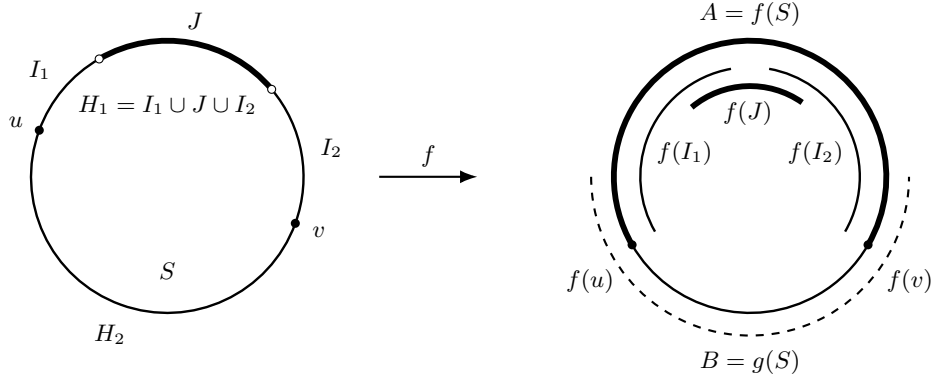

For $g$, choose $u',v'\in S$ mapped by $g$ to the two endpoints of $B$. The closures of
the two components of $S\setminus\{u',v'\}$ are both mapped onto $B$. Hence
Lemma~\ref{lem:nonexp} and Lemma~\ref{lem:cut} with $M=S$ and $Z=\{u',v'\}$ give
\begin{equation}\label{eq:5.4}
2N(B)\le N(S)+4.
\end{equation}

Since $A\cup B=S$, 
Lemma~\ref{lem:restr} yields
\begin{equation}\label{eq:5.5}
N(S)\le N(A)+N(B).
\end{equation}
Combining \eqref{eq:5.3}--\eqref{eq:5.5},
\[
2N(S)\le2N(A)+2N(B)\le2N(S)+12-\Delta.
\]
Thus $\Delta\le12$, proving \eqref{eq:5.1}.
\end{proof}

\section{Proof of the main theorem}\label{sectionmain}

\begin{proof}[Proof of Theorem~\ref{thm:main}]
Suppose, towards a contradiction, that $\F=\{f,g\}$ is topologically contractive and
that $f(S)\cup g(S)=S$. By Lemma~\ref{lem:metric} there is a compatible metric $\rho$, under which both maps $f$ and $g$ are nonexpanding.

Clearly, neither $f$ nor $g$ is surjective and thus also neither is constant. Therefore
\[
A=f(S),\qquad B=g(S)
\]
are nondegenerate proper closed arcs.

Choose $u,v\in S$ mapped by $f$ to the two endpoints of $A$.
By applying Lemma~\ref{lem:loss}
we get that $f$ cannot be constant on any nondegenerate arc. Otherwise there
would be a nondegenerate subarc $J$ whose interior avoids $u$ and $v$, and
\eqref{eq:5.1} would imply
\[
N_\varepsilon(J)\le N_\varepsilon(f(J))+12=13\qquad\text{for every }\varepsilon>0,
\]
contradicting Lemma~\ref{lem:end}.

Set
\[
J_n=f^n(S),\qquad n\ge1.
\]
The arcs $J_n$ are nondegenerate by the preceding paragraph. They are nested, since
\[
f^{n+1}(S)=f^n(f(S))\subseteq f^n(S),
\]
and their $\rho$-diameters tend to zero.
Hence
\begin{equation}\label{eq:6.1}
\bigcap_{n\ge1}J_n=\{p\}\ \text{ for some } p\in S,\qquad f(p)=p.
\end{equation}
Every inclusion $J_{n+1}\subseteq J_n$ is strict; equality at one stage would make all
subsequent images equal to the same nondegenerate arc.

For all sufficiently large $n$,
\begin{equation}\label{eq:6.2}
\inte_S(J_n)\cap\{u,v\}=\emptyset.
\end{equation}
Indeed, any point of $\{u,v\}$ different from $p$ is eventually outside $J_n$, by
\eqref{eq:6.1} and nesting. If, for instance, $u=p$, then since $f(u)$ is an endpoint
of $A$ and $f(u)=f(p)=p$, we conclude that the point $p$
is an endpoint of $J_n$ and not an interior point. 
The same reasoning applies if $v=p$.

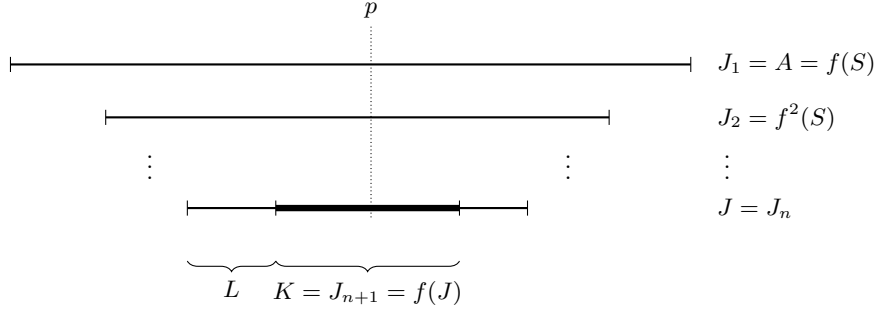
\begin{figure}[htb]
\centering
\begin{tikzpicture}[x=0.9cm,y=1cm,font=\small]
  \draw[densely dotted] (5.3,0.5) -- (5.3,-2.05);
  \node[above] at (5.3,0.5) {$p$};
  \draw[thick] (0,0) -- (10,0);
  \draw (0,0.1) -- (0,-0.1) (10,0.1) -- (10,-0.1);
  \node[right] at (10.25,0) {$J_1=A=f(S)$};
  \draw[thick] (1.4,-0.7) -- (8.8,-0.7);
  \draw (1.4,-0.6) -- (1.4,-0.8) (8.8,-0.6) -- (8.8,-0.8);
  \node[right] at (10.25,-0.7) {$J_2=f^2(S)$};
  \node at (2.05,-1.25) {$\vdots$};
  \node at (8.2,-1.25) {$\vdots$};
  \node[right] at (10.35,-1.25) {$\vdots$};
  \draw[thick] (2.6,-1.9) -- (7.6,-1.9);
  \draw[line width=2.4pt] (3.9,-1.9) -- (6.6,-1.9);
  \foreach \x in {2.6,3.9,6.6,7.6} \draw (\x,-1.8) -- (\x,-2.0);
  \node[right] at (10.25,-1.9) {$J=J_n$};
  \draw[decorate,decoration={brace,mirror,amplitude=4pt}] (2.6,-2.6) -- (3.9,-2.6)
     node[midway,below=4pt] {$L$};
  \draw[decorate,decoration={brace,mirror,amplitude=4pt}] (3.9,-2.6) -- (6.6,-2.6)
     node[midway,below=4pt] {$K=J_{n+1}=f(J)$};
\end{tikzpicture}
\caption{The arcs $J=J_n$, $K=f(J)$, and $L$.}
\label{fig:end}
\end{figure}

Fix an $n$ satisfying \eqref{eq:6.2}, and put
\[
J=J_n,\qquad K=J_{n+1}=f(J).
\]
Then $K\subsetneq J$ and both arcs are nondegenerate. 
Then $J\setminus K$ has at least one (and also at most two) nondegenerate component.
Let $L$ be the closure of such a component, so it is a nondegenerate arc (see Figure~\ref{fig:end}). Then $L$ intersects $K$ in their common endpoint, and $L\cup K\subseteq J$. By Lemma \ref{lem:cut} and
Lemma~\ref{lem:restr},
\[
N_\varepsilon(L)+N_\varepsilon(K)\le N_\varepsilon(L\cup K)+4\le N_\varepsilon(J)+4.
\]
Since $K=f(J)$, Lemma~\ref{lem:loss} yields
\begin{equation}\label{eq:6.3}
N_\varepsilon(L)\le N_\varepsilon(J)-N_\varepsilon(f(J))+4\le12+4=16
\qquad\text{for every }\varepsilon>0.
\end{equation}
The arc $L$ is fixed independently of $\varepsilon$. If $a,b$ are distinct
endpoints of $L$, then Lemma~\ref{lem:end} and \eqref{eq:6.3} give
\[
0<\rho(a,b)\le\varepsilon N_\varepsilon(L)\le16\varepsilon\qquad\text{for every }
\varepsilon>0,
\]
which is a contradiction. Thus the theorem is proved.
\end{proof}

\section{Questions}

Let us recall Question 29 from \cite{KarasovaVejnar}, whose case $n=1$ is settled by the present paper. The number $n+2$ is suggested by the Lusternik--Schnirelmann theorem (see the paper for a discussion).

\begin{question}\label{q:spheres}
What is $\wn(S^n)$ for every $n\ge2$?
\end{question}

It is easy to see that $\wn(S^n)\leq n+2$. However, we still do not know whether $\wn(S^n)\geq 3$ for $n\geq 2$.

%\begin{question} What topological graphs $G$ satisfy $\wn(G)=2$? \end{question}

\section{Acknowledgement}
The author acknowledges the use of artificial intelligence tools in obtaining the results of this paper. The work developed over a period of approximately six months, beginning with numerical experiments and heuristic estimates, followed by attempts to resolve the main question for progressively broader special classes of maps, which ultimately led to the general results presented here.
The author takes full responsibility for these results.

\bibliographystyle{abbrv}
\bibliography{citace}

\end{document}